\documentclass[11pt,reqno]{amsart}
\usepackage[T1]{fontenc}
\usepackage[utf8]{inputenc}
\usepackage[a4paper,margin=2.8cm]{geometry}
\usepackage{amsmath,amssymb,amsfonts,amsthm,mathtools}
\usepackage{enumitem}
\usepackage{microtype}
\usepackage{xcolor}
\usepackage[hidelinks,bookmarksdepth=3]{hyperref}
\usepackage[nameinlink,noabbrev]{cleveref}
\numberwithin{equation}{section}
\theoremstyle{plain}
\newtheorem{theorem}{Theorem}[section]
\newtheorem{lemma}[theorem]{Lemma}
\newtheorem{proposition}[theorem]{Proposition}

\theoremstyle{definition}

\newtheorem{remark}[theorem]{Remark}
\newtheorem{question}[theorem]{Question}
\crefname{theorem}{Theorem}{Theorems}
\crefname{lemma}{Lemma}{Lemmas}
\crefname{proposition}{Proposition}{Propositions}
\crefname{corollary}{Corollary}{Corollaries}
\crefname{remark}{Remark}{Remarks}
\crefname{question}{Question}{Questions}

\title[Counting almost independent sets in regular graphs]
{Counting almost independent sets in regular graphs}
\author[G. Carenini]{Gaia Carenini}
\address{Trinity College Cambridge, Department of Pure Mathematics and Mathematical Statistics, Centre for Mathematical Sciences, Wilberforce Road, Cambridge CB3 0WA, United Kingdom.}
\email{gc645@cam.ac.uk}
\date{\today}

\begin{document}

\begin{abstract}
Kahn proved that, among bipartite $d$-regular graphs on $n$ vertices, the number of independent sets is maximized by a disjoint union of copies of $K_{d,d}$. Zhao later extended this result to all $d$-regular graphs. We prove a robust version of this theorem in which independent sets are replaced by sets spanning few internal edges. If $G$ is $d$-regular on $n$ vertices, then the number of subsets spanning at most $\gamma dn$ edges is at most
$$
2^{n/2}\exp\left\{
O\bigl(\gamma\log(e/\gamma)n\bigr)
+
O\bigl(n/d\bigr)
\right\}.
$$
Both correction terms are sharp up to absolute constants: the $\gamma\log(1/\gamma)n$ term is necessary when $d$ is sufficiently large in terms of $\gamma$, while the $n/d$ term is already necessary for independent sets. Our result answers a question of Seth.
\end{abstract}

\maketitle
\section{Introduction}

How many subsets of a regular graph can be independent? This question has a well-known history. Motivated by the problem of counting sum-free sets, Granville conjectured that every $n$-vertex $d$-regular graph has at most $2^{(1/2+o(1))n}$ independent sets as $d\to\infty$. This was proved by Alon \cite{Alon}, and later sharpened to an exact extremal theorem through the work of Kahn \cite{Kahn} and Zhao \cite{Zhao}. Kahn proved the sharp bound for bipartite $d$-regular graphs: if $G$ is bipartite and $d$-regular on $n$ vertices, then the number of independent sets of $G$ is at most $(2^{d+1}-1)^{n/(2d)}$. Zhao subsequently removed the bipartiteness assumption, proving that the same bound holds for every $d$-regular graph. Thus, whenever $2d\mid n$, the number of independent sets is maximized by a disjoint union of copies of $K_{d,d}$.

A related counting problem was studied by Seth \cite{Seth} in the context of tolerant independent-set testing. He proved that, for $k\ge c\log^9 n$, the number of subsets $J\subseteq V(G)$ satisfying
$\frac{e_G(J)}{\binom{|J|}{2}}\le \frac{d}{kn}
$
is at most
$$
2^{\frac n2\left(
1+O\!\left(\frac{\log^3n}{d}\right)
 +O\!\left(\frac{\log^3n}{k^{1/3}}\right)
\right)}.
$$
Seth further asked for the correct order of this bound, and whether the disjoint union of copies of $K_{d,d}$ is extremal.

In this paper, we address the quantitative part of this question by proving a robust form of the Kahn--Zhao theorem. Instead of counting only independent sets, we count all sets spanning at most $\gamma dn$ edges, for a parameter $\gamma\ge0$ interpolating between the independent-set count ($\gamma=0$) and increasingly sparse induced subgraphs. We show that the extremal $2^{n/2}$ behaviour persists for small $\gamma$, at the cost of two correction terms whose orders we determine up to absolute constants.

We shall use the following notation. If $G$ is a graph and $A\subseteq V(G)$, write $e_G(A)=e(G[A])$. For a $d$-regular graph $G$ on $n$ vertices and $\gamma\ge0$, set
$$
i_\gamma(G):=\bigl|\{A\subseteq V(G):e_G(A)\le\gamma dn\}\bigr|.
$$
Thus $i_0(G)=i(G)$, the number of independent sets. We can now state our main theorem.

\begin{theorem}\label{thm:main}
There is an absolute constant $C>0$ such that the following holds. Let $G$ be a $d$-regular graph on $n$ vertices, with $d\ge2$, and let $0<\gamma\le1/2$. Then
$$
i_\gamma(G)
\le
2^{n/2}
\exp\left\{
C\gamma\log(e/\gamma)n
+
C\frac nd
\right\}.
$$
\end{theorem}

The error terms in \Cref{thm:main} are unavoidable, up to absolute constants.

\begin{proposition}\label{prop:sharpness}
There are absolute constants $c,c_0>0$ such that the following holds. Suppose that $0<\gamma\le c_0$, $d\ge2/\gamma$, and $2d\mid n$. If $G$ is the disjoint union of $n/(2d)$ copies of $K_{d,d}$, then
$$
i_\gamma(G)
\ge
2^{n/2}
\exp\left\{
c\gamma\log(e/\gamma)n
+
c\frac nd
\right\}.
$$
Moreover, for every $d\ge2$ and $2d\mid n$,
$$
i_0(G)
=
(2^{d+1}-1)^{n/(2d)}
\ge
2^{n/2}\exp\left\{c\frac nd\right\}.
$$
\end{proposition}

Thus the $\gamma\log(1/\gamma)n$ term has the correct order once $\gamma d$ is bounded below, while the $n/d$ correction is already forced at $\gamma=0$.

Seth's density condition implies
$$
e_G(J)
\le
\frac{d}{kn}\binom{|J|}{2}
\le
\frac{dn}{2k}.
$$
Thus, taking $\gamma=1/(2k)$ in \Cref{thm:main} gives
$$
2^{n/2}
\exp\left\{
O\!\left(\frac{\log k}{k}n\right)
+
O\!\left(\frac nd\right)
\right\},
$$
matching, up to absolute constants, the dependence exhibited by Seth's $K_{d,d}$ construction. In this sense, \Cref{thm:main} answers the quantitative part of Seth's question. A different sparse-induced-subgraph counting problem was studied by Nenadov \cite{Nenadov}, for locally dense host graphs under a maximum-degree condition on the induced subgraph.

The proof of \Cref{thm:main} reduces to a single idea. Rather than count sparse sets directly, we penalize each induced edge by a factor $q<1$, turning $i_\gamma(G)$ into an exponential moment of the partition function of an antiferromagnetic two-spin model. A theorem of Sah, Sawhney, Stoner and Zhao \cite{SSSZ} shows that among $d$-regular graphs, every such partition function is maximized by $K_{d,d}$. This reduces the graph-theoretic problem to estimating a single weighted block: the partition function of the same model on one copy of $K_{d,d}$, where choosing $a$ vertices on one side and $b$ on the other is penalized by $q^{ab}$ for spanning $ab$ edges. The reduction works for an arbitrary vertex fugacity $\lambda$, and we record that weighted form in \Cref{prop:biclique-reduction} below.

The exact extremal problem remains natural.

\begin{question}\label{q:exact}
Suppose that $2d\mid n$. For every $\gamma\ge0$, is $i_\gamma(G)$ maximized, among $d$-regular graphs on $n$ vertices, by the disjoint union of $n/(2d)$ copies of $K_{d,d}$?
\end{question}

\section{Proofs}

All logarithms are natural and all implicit constants are absolute. For a $d$-regular graph $G$ on $n$ vertices, $0<q\le1$ and $\lambda>0$, define
$$
Z_G(\lambda,q)
=
\sum_{A\subseteq V(G)} \lambda^{|A|}q^{e_G(A)},
\qquad\text{so that}\qquad
Z_{K_{d,d}}(\lambda,q)
=
\sum_{a,b=0}^d
\binom da\binom db
\lambda^{a+b}q^{ab}.
$$

\begin{proposition}[Biclique reduction]\label{prop:biclique-reduction}
Let $G$ be a $d$-regular graph on $n$ vertices, let $\gamma\ge0$, and let $\lambda>0$. Then, for every $t\ge0$,
$$
\sum_{\substack{A\subseteq V(G)\\ e_G(A)\le\gamma dn}}
\lambda^{|A|}
\le
e^{\gamma tn}
Z_{K_{d,d}}\!\left(\lambda,e^{-t/d}\right)^{n/(2d)}.
$$
\end{proposition}

\begin{proof}
We first bound $Z_G(\lambda,q)$ itself for $0<q\le1$. Consider the two-spin model with spins $\{0,1\}$, vertex weights $w(0)=1$, $w(1)=\lambda$, and edge-interaction matrix
$$
W=
\begin{pmatrix}
1 & 1\\
1 & q
\end{pmatrix}.
$$
For a spin assignment $\sigma:V(G)\to\{0,1\}$, the contribution is
$$
\prod_{v\in V(G)} w(\sigma(v))
\prod_{uv\in E(G)} W(\sigma(u),\sigma(v)).
$$
If $A=\{v\in V(G):\sigma(v)=1\}$, this contribution is exactly $\lambda^{|A|}q^{e_G(A)}$. Consequently, the partition function of this two-spin model on $G$ is $Z_G(\lambda,q)$. The model is antiferromagnetic, since
$$
W(0,0)W(1,1)=q\le1=W(0,1)W(1,0).
$$
Corollary~1.14 of Sah, Sawhney, Stoner and Zhao \cite{SSSZ} states that every antiferromagnetic two-spin model is biclique-maximizing. Applied to a $d$-regular graph $G$, this gives
$$
Z_G(\lambda,q)
\le
Z_{K_{d,d}}(\lambda,q)^{n/(2d)},
$$
where, in $K_{d,d}$, a set choosing $a$ vertices on one side and $b$ vertices on the other has weight $\lambda^{a+b}q^{ab}$, and there are $\binom da\binom db$ such choices, giving the formula for $Z_{K_{d,d}}(\lambda,q)$ above.

Now put $q=e^{-t/d}$. If $e_G(A)\le\gamma dn$, then
$q^{e_G(A)}\ge q^{\gamma dn}=e^{-\gamma tn}$.
Hence
$$
\sum_{\substack{A\subseteq V(G)\\ e_G(A)\le\gamma dn}}
\lambda^{|A|}
\le
e^{\gamma tn}
\sum_{A\subseteq V(G)}
\lambda^{|A|}q^{e_G(A)}
=
e^{\gamma tn}Z_G(\lambda,q)
\le
e^{\gamma tn}
Z_{K_{d,d}}\!\left(\lambda,e^{-t/d}\right)^{n/(2d)}.
\qedhere
$$
\end{proof}

\begin{remark}
The fugacity $\lambda$ can be used to extract bounds for almost-independent sets of a prescribed size. For example, if $N_m$ denotes the number of sets $A$ with $|A|=m$ and $e_G(A)\le\gamma dn$, then \Cref{prop:biclique-reduction} gives, for every $\lambda>0$ and $t\ge0$,
$$
N_m
\le
\lambda^{-m}e^{\gamma tn}
Z_{K_{d,d}}\!\left(\lambda,e^{-t/d}\right)^{n/(2d)}.
$$
We only use $\lambda=1$ below.
\end{remark}

With $\lambda=1$, \Cref{prop:biclique-reduction} reduces the problem to
$$
B_d(t):=\sum_{a,b=0}^{d}\binom da\binom db e^{-tab/d}.
$$
Roughly speaking, $B_d(t)$ weighs every pair of subsets $(A,B)$ of the two sides of $K_{d,d}$ by $e^{-t\cdot(e(A,B))/d}$ where $e(A,B)$ denotes the number of edges between $A$ and $B$, so that configurations spanning many edges are exponentially suppressed; the parameter $t$ controls how strongly cross-edges are penalized, with larger $\gamma$ (a looser edge budget) corresponding to smaller $t$. In $K_{d,d}$, choosing $a$ and $b$ vertices on the two sides spans $ab$ edges. Taking $t=2\log(e/\gamma)$ balances entropy against the edge penalty. The estimate below shows that a significant contribution to $B_d(t)$ comes only when one side is chosen almost freely and the other contains a small sprinkling of vertices. This gives the main factor $2^{n/2}$, the entropy correction $\exp\{O(\gamma\log(e/\gamma)n)\}$, and the finite-degree correction $\exp\{O(n/d)\}$.

\begin{lemma}\label{lem:block-estimate}
There is an absolute constant $C_0>0$ such that, for all $d\ge2$ and $t\ge1$,
$$
B_d(t)
=
\sum_{a,b=0}^{d}\binom da\binom db e^{-tab/d}
\le
2^d
\exp\left\{
C_0dt e^{-t/2}
+
C_0
\right\}.
$$
\end{lemma}

\begin{proof}
For $1\le t\le20$ the estimate follows from the trivial bound
$B_d(t)\le4^d,
$
after increasing $C_0$, since $t e^{-t/2}$ is bounded away from zero on this interval. We may therefore assume that $t\ge20$. Let
$r=e^{-t/2}$.
First summing over $b$ gives
$$
B_d(t)=\sum_{a=0}^{d}\binom da(1+e^{-ta/d})^d.
$$
Although the original expression is symmetric in $a$ and $b$, this asymmetric form, that we obtained by explicitly evaluating the geometric-type sum over $b$, turns the problem into a one-dimensional entropy estimate. We bound this sum by splitting the range of $a/d$ into scales and applying a different estimate on each range, the standard \emph{divide-and-conquer} approach for binomially weighted exponential sums. We shall repeatedly use the standard Chernoff-type entropy tail bound for binomial coefficients
\begin{equation}\label{eq:entropy-tail}
\sum_{a\le pd}\binom da\le \exp\{dh(p)+C\}
\qquad (0\le p\le1/2)
\end{equation}
and the following de~Moivre--Laplace-type approximation:
\begin{equation}\label{eq:local-binomial}
\binom da
\le
\frac{C}{\sqrt{1+dx(1-x)}}\exp\{dh(x)\},
\qquad x=a/d,
\end{equation}
with an absolute constant $C$, and $h(x):=-x\log x-(1-x)\log(1-x)$, with $0\log0=0$. We also use
\begin{equation}\label{eq:entropy-basic}
h(x)\le x\log(e/x),
\qquad
h(x)\le \log2-2(x-1/2)^2.
\end{equation}

First consider $a/d\le4r$. Since $(1+e^{-ta/d})^d\le2^d$, the entropy tail bound \eqref{eq:entropy-tail} and the bound $h(4r)\le Ctr$ show that this range contributes at most $2^d\exp\{Cdtr+C\}$. Next consider $a/d\ge1/2$. In this range
$(1+e^{-ta/d})^d\le \exp\{dr\}$, 
and hence the contribution is at most
$2^d\exp\{dr\}$. It remains to estimate the middle range $4r<a/d<1/2$. Let $x=a/d$ and
$$
\Phi_t(x):=h(x)+\log(1+e^{-tx})-\log2.
$$

For $4r\le x\le2/t$, let $u=tx$. Using a second-order Taylor expansion of $\log(1+e^{-u})$ around $u=0$, together with \eqref{eq:entropy-basic}, namely
$\log(1+e^{-u})\le \log2-\frac u2+\frac{u^2}{8}$,
we have
$$
\Phi_t(x)
\le
x\left(\log\frac{et}{u}-\frac t2+\frac{tu}8\right).
$$
The expression in parentheses is convex as a function of $u$ on $[4tr,2]$, so its maximum on this interval is attained at an endpoint. At the endpoints it is at most an absolute negative constant, since
$$
\log\frac{e}{4r}-\frac t2+\frac{t(4tr)}8
=
1-\log4+\frac{t^2r}{2}<0
$$
and
$$
\log\frac{et}{2}-\frac t2+\frac t4
=
1+\log(t/2)-\frac t4<0
$$
for $t\ge20$. Thus there is an absolute constant $c>0$ such that
$$
\Phi_t(x)\le-cx
\qquad(4r\le x\le2/t).
$$
Using the local binomial approximation \eqref{eq:local-binomial}, and noting that
$d x(1-x)\ge a/2$ here, the contribution from $4r<a/d\le2/t$ is at most
$$
C2^d
\sum_{a\ge0}
\frac{e^{-ca}}{\sqrt{1+a}}
\le
C2^d,
$$
a convergent geometric-type tail sum.

Now consider
$\frac2t\le x\le\frac12-\frac2t
$
and write $y=1/2-x$. Then
$$
\frac2t\le y\le\frac12-\frac2t.
$$
The function $e^{-t/2+ty}/y^2$ is increasing on this interval, and therefore
$$
e^{-tx}=e^{-t/2+ty}\le y^2.
$$
Using \eqref{eq:entropy-basic} and $\log(1+e^{-tx})\le e^{-tx}$ gives
$$
\Phi_t(x)\le-y^2.
$$

We split this range once more. If $2/t\le x\le1/4$, then $y\ge1/4$, and the crude bound $\binom da\le e^{dh(x)}$ gives
$$
\binom da(1+e^{-tx})^d
\le
2^d e^{-d/16}.
$$
Summing over at most $d+1$ values of $a$, this part contributes $O(2^d)$.

If instead
$$
\frac14\le x\le\frac12-\frac2t,
$$
then $x(1-x)\ge3/16$. This is a local-CLT range: the local estimate \eqref{eq:local-binomial} yields a genuinely Gaussian-shaped bound,
$$
\binom da(1+e^{-tx})^d
\le
\frac{C2^d}{\sqrt d}
\exp\left\{-\frac{(d/2-a)^2}{d}\right\},
$$
and summing this Gaussian profile over $a$, a standard $\sum_a e^{-a^2/d}=O(\sqrt d)$ estimate, shows that this part contributes at most
$
C2^d.
$

Finally, if $1/2-2/t\le x\le1/2$ (the range nearest the peak, again treated by the local-CLT approximation), then
$e^{-tx}\le e^2r
$
and hence
$$
\Phi_t(x)
\le
-2(x-1/2)^2+e^2r.
$$
Since $x\ge1/2-2/t\ge2/5$ for $t\ge20$, \eqref{eq:local-binomial} gives
$$
\binom da(1+e^{-tx})^d
\le
\frac{C2^d}{\sqrt d}
\exp\{Ce^2dr\}
\exp\left\{-\frac{2(d/2-a)^2}{d}\right\},
$$
and the same Gaussian-sum estimate as above shows that this range contributes at most
$C2^d\exp\{Cdr\}$.

Combining the estimates for all ranges, and absorbing the leading absolute constants into the exponential, proves the lemma.
\end{proof}

\begin{proof}[Proof of \Cref{thm:main}]
Set
$t=2\log(e/\gamma)$.
Since $0<\gamma\le1/2$, we have $t\ge1$. Applying \Cref{prop:biclique-reduction} with $\lambda=1$ gives
\begin{equation}\label{eq:reduction-main}
i_\gamma(G)
\le
e^{\gamma tn}
B_d(t)^{n/(2d)}.
\end{equation}
Now
$$
t e^{-t/2}
=
2\log(e/\gamma)\cdot\frac{\gamma}{e}
=
O\bigl(\gamma\log(e/\gamma)\bigr).
$$
By Lemma \ref{lem:block-estimate},
$$
B_d(t)
\le
2^d
\exp\left\{
Cd\gamma\log(e/\gamma)+C
\right\}.
$$
Substituting this into \eqref{eq:reduction-main} gives
$$
i_\gamma(G)
\le
2^{n/2}
\exp\left\{
C\gamma\log(e/\gamma)n
+
C\frac nd
\right\},
$$
after increasing $C$ if necessary.
\end{proof}

\begin{proof}[Proof of Proposition \ref{prop:sharpness}]
Let
$m=\frac{n}{2d}
$
be the number of $K_{d,d}$ blocks.

First suppose that $0<\gamma\le c_0$, where $c_0>0$ is a sufficiently small absolute constant, and that $d\ge2/\gamma$. Set
$k=\lfloor\gamma d\rfloor,
\qquad
S_k=\sum_{j=0}^{k}\binom dj$.
In each copy of $K_{d,d}$, consider all vertex sets for which at least one of the two sides contains at most $k$ selected vertices. Every such set spans at most $kd\le\gamma d^2$ edges. Therefore an arbitrary choice from this family in every block spans at most
$m\gamma d^2
=
\frac{\gamma dn}{2}
\le
\gamma dn$
edges in total.

The number of allowed sets in one block is
$$
2^{d+1}S_k-S_k^2
=
2^dS_k\left(2-\frac{S_k}{2^d}\right).
$$
Taking $c_0\le1/4$, we have $k<d/2$, and hence $S_k\le2^{d-1}$. Thus the one-block count is at least
$\frac32\,2^dS_k$.
It follows that
$$
i_\gamma(G)
\ge
\left(\frac32\,2^dS_k\right)^m
=
2^{n/2}
\exp\left\{
\frac{n}{2d}
\left(
\log\frac32+\log S_k
\right)
\right\}.
$$

Since $\gamma d\ge2$, we have $k\ge\gamma d/2$. Moreover,
$$
S_k\ge\binom dk
\ge
\left(\frac dk\right)^k.
$$
As $k\le\gamma d$, this gives
$$
\log S_k
\ge
k\log\frac dk
\ge
\frac{\gamma d}{2}\log\frac1\gamma.
$$
After decreasing $c_0$ if necessary,
$$
\log\frac1\gamma
\ge
\frac12\log\frac e\gamma.
$$
Hence
$$
\frac{n}{2d}\log S_k
\ge
c\gamma\log(e/\gamma)n,
$$
while
$$
\frac{n}{2d}\log\frac32
\ge
c\frac nd.
$$
This proves the first assertion.

For $\gamma=0$, each $K_{d,d}$ has exactly $2^{d+1}-1$ independent sets, and therefore
$$
i_0(G)
=
(2^{d+1}-1)^m
=
2^{n/2}
\exp\left\{
\frac{n}{2d}\log(2-2^{-d})
\right\}.
$$
Since $d\ge2$, the last logarithm is bounded below by a positive absolute constant, which proves the second assertion.
\end{proof}
\subsection*{Use of generative AI tools}
Generative AI tools were not used to generate the mathematical results or underlying ideas. They were used solely for language editing, including correcting grammar and spelling and improving readability; for identifying potentially relevant literature, including the article by Nenadov \cite{Nenadov}; and for standardizing the bibliography entries.

\subsection*{Acknowledgments}
The author thanks Cameron Seth for several discussions on counting almost independent sets in graphs and hypergraphs. The author is also grateful to her supervisor, Imre Leader, for his support and guidance. The author is supported by the CB European PhD Studentship funded by Trinity College, Cambridge.

\end{document}